\documentclass[reqno]{amsart}
\usepackage{float}
\usepackage{amsmath} 
\usepackage{graphicx} 
\usepackage{latexsym}
\usepackage{amsfonts}
\usepackage{amssymb}
\usepackage{color} 
\usepackage[normalem]{ulem}

\theoremstyle{plain}
\newtheorem{theorem}{Theorem}
\newtheorem{corollary}[theorem]{Corollary}
\newtheorem{lemma}{Lemma}

\theoremstyle{definition}

\newtheorem{example}[theorem]{Example}

\newtheorem*{remark*}{Remark}

\renewcommand{\Pr}{\mathbf P}
\newcommand{\E}{\mathbf E}
\newcommand{\ind}[1]{{\rm \bf1}\{#1\}}

\newcommand{\ee}{\varepsilon}
\newcommand{\ff}{\varphi}
\newcommand{\kk}{\varkappa}

\newcommand{\uu}{\underline}

\begin{document}
\title[First-passage times for bounded increments]
{First-passage times for random walks
in the triangular array setting }
\author[Denisov]{Denis Denisov}
\address{Department of Mathematics, University of Manchester, Oxford Road, Manchester M13 9PL, UK}
\email{denis.denisov@manchester.ac.uk}

\author[Sakhanenko]{Alexander Sakhanenko}
\address{Sobolev Institute of Mathematics, 630090 Novosibirsk, Russia}
\email{aisakh@mail.ru}

\author[Wachtel]{Vitali Wachtel}
\address{Institut f\"ur Mathematik, Universit\"at Augsburg, 86135 Augsburg, Germany}
\email{vitali.wachtel@math.uni-augsburg.de}

\begin{abstract}
In this paper we continue our study of exit times for random walks with independent but not necessarily identically distributed increments. 
Our paper ``First-passage times for random walks with non-identically distributed increments'' (2018) was devoted to the case when the random walk is constructed by a fixed sequence of independent random variables which satisfies the classical Lindeberg condition. 
Now we consider a more general situation when we have a triangular array of independent random variables.
Our main assumption is that the entries of every row are uniformly bounded by a deterministic sequence, 
which tends  to zero as the number of the row increases.  
\end{abstract}


\thanks{\rm 
A.S. and V.W. were supporteed by RFBR and DFG according to the research project № 20-51-12007}
\maketitle
{\scriptsize
}

\section{Introduction and the main result.}
\subsection{Introduction}
Suppose that for each $n=1,2,\dots $ we are given independent random variables 
$X_{1,n},\dots,X_{n,n}$ such that 
\begin{gather}                                                                               \label{i1}
\E X_{i,n}=0
\quad\text{for all }i\le n
\qquad\text{and}\qquad
\sum_{i=1}^{n}\E X_{i,n}^2=1.
\end{gather}
For each $n$ we consider a random walk
\begin{gather}                                                                               \label{i2}
S_{k,n}:=X_{1,n}+\dots+X_{k,n},\quad k=1,2,\dots,n.
\end{gather}
Let
$\{g_{k,n}\}_{k=1}^n$   be  deterministic real 
numbers, and let
\begin{gather}                                                                               \label{i3}
T_n:=\inf\{k\geq1:S_{k,n}\leq g_{k,n}\}
\end{gather}
be the first crossing over the moving boundary $\{g_{k,n}\}$ by the random walk $\{S_{k,n}\}$.
The main purpose of the present paper is to study the asymptotic behaviour,
as $n\to\infty$, of the probability
\begin{gather}
\label{i4}
\Pr(T_n>n)=\Pr\left(\min_{1\le k\le n}(S_{k,n}-g_{k,n})>0\right).
\end{gather}

We shall always assume that the boundary $\{g_{k,n}\}$ is of a small magnitude, that is,
\begin{gather}                                                                                          \label{i6}
g_n^*:=\max_{1\le k\le n}|g_{k,n}|\to0.
\end{gather}
Here and in what follows, all unspecified limits are taken with respect to
$n\to\infty$.

Furthermore, to avoid trivialities, we shall assume that
\begin{gather}                                                                                        \label{i7}
\mathbf{P}(T_n>n)>0\qquad\text{for all }\quad n\ge 1.
\end{gather}

An important  particular case of the triangular array scheme is given by the following construction. Let $X_1,X_2,\dots $ be independent random variables with finite variances such that
\begin{gather}                                                                               \label{i11}
\E X_{i}=0\quad \text{ for all }i\ge1
\qquad\text{and}\qquad
B_n^2:=\sum_{i=1}^{n}\E X_{i}^2\to\infty.
\end{gather}
For  a real deterministic sequence $\{g_1,g_2,\dots\}$ set
\begin{gather}                                                                               \label{i13}
T:=\inf\{k\geq1:S_{k}\leq g_{k}\},
\quad\text{where}\quad
S_{k}:=X_{1}+\dots+X_{k}.
\end{gather}
Stopping time  $T$ is the first crossing over the moving boundary $\{g_{k}\}$ by the random walk~$\{S_{k}\}$.
Clearly,  \eqref{i11} --  \eqref{i13} is a particular case of  \eqref{i1} --  \eqref{i3}. Indeed to obtain~\mbox{ \eqref{i1} --  \eqref{i3}}  it is sufficient to set
\begin{gather}                                                                                          \label{i15}
X_{k,n}=\frac{X_k}{B_n},\quad S_{k,n}=\frac{S_k}{B_n},\quad g_{k,n}=\frac{g_k}{B_n}.
\end{gather}
However, the triangular array scheme is much more general than \eqref{i11} --  \eqref{i15}. 

If the classical Lindeberg condition holds for the sequence $\{X_k\}$
and $g_n=o(B_n)$ then, according to Theorem 1 in \cite{DSW16},
\begin{equation}                                                                                          \label{i17}
\Pr(T>n)\sim \sqrt{\frac{2}{\pi}}\frac{U(B_n^2)}{B_n}, 
\end{equation}
where $U$ is a positive slowly varying function with the values
$$
U(B_n^2)=\E[S_n-g_n;T>n],\quad n\ge1.
$$

The constant \(\sqrt{\frac{2}{\pi}}\) in front of the asymptotics 
has been inherited from the tail asymptotics of exit time of standard Brownian motion.  
Indeed, let $W(t)$ be the standard Brownian motion and set
$$
\tau_x^{bm}:=\inf\{t>0:x+W(t)\le0\},\quad x>0.
$$
Then, 
$$
\Pr(\tau_x^{bm}>t)=
\Pr(|W(t)|\le x)
=\Pr\left(|W(1)|\le \frac x{\sqrt t}\right)
\sim \sqrt{\frac{2}{\pi}}\frac{x}{\sqrt{t}},
\quad \text{as}\quad \frac{x}{\sqrt{t}}\to 0.
$$
The continuity of paths of $W(t)$ implies that $x+W(\tau_x^{bm})=0$.
Combining this with the optional stopping theorem, we obtain
\begin{align*}
x=\E[x+W(\tau_x^{bm}\wedge t)]&=\E[x+W(t);\tau_x^{bm}>t)]+
\E[x+W(\tau_x^{bm});\tau_x^{bm}\le t)]\\
&=\E[x+W(t);\tau_x^{bm}>t)].
\end{align*}
Therefore, for any fixed $x>0$, 
$$
\Pr(\tau_x^{bm}>t)
\sim\sqrt{\frac{2}{\pi}}\frac{x}{\sqrt{t}}= \sqrt{\frac{2}{\pi}}\frac{\E[x+W(t);\tau_x^{bm}>t)]}{\sqrt{t}},\quad \text{as}\quad  t\to\infty.
$$
Thus, the right hand sides here and in \eqref{i17} are of the same type.

\subsection{Main result}
The purpose of the present note is to generalise the asymptotic relation \eqref{i17} to the triangular array setting. 
More precisely, we are going to  show that the following relation holds
\begin{gather}                                                                            \label{i20}
\Pr(T_n>n)\sim\sqrt{\frac{2}{\pi}}E_n,
\end{gather}
where
\begin{gather}                                                                            \label{i20+}
E_n:=\E[S_{n,n}-g_{n,n};T_n>n]
=\E[-S_{T_n,n};T_n\le n]-g_{n,n}\Pr(T_n>n).
\end{gather}

Unexpectedly  for the authors, in contrast to the described above case of a single 
sequence, the Lindeberg condition is not sufficient for the validity of \eqref{i20}, 
see Example~\ref{Lind}. 
Thus, one has to find a more restrictive condition for \eqref{i20} to hold. 
In this paper we show that \eqref{i20} holds under the following assumption: 
there exists a sequence $r_n$ such that
\begin{gather}                                                                                          \label{i5}
\max_{1\le i\le n}|X_{i,n}|\le r_n\to 0.
\end{gather}
It is clear that under this  assumption the triangular array satisfies 
the Lindeberg condition and, hence, the Central Limit Theorem holds.

At  first glance, \eqref{i5} might look too restrictive. 
However we shall construct a triangular array, see Example~\ref{Lind2}, 
in which the assumption \eqref{i5} becomes necessary for \eqref{i20} to hold.
Now we  state our main result.

\begin{theorem}                                                                      \label{thm:main}
Assume that \eqref{i6} and  \eqref{i5} are valid. Then there exists an absolute constant $C_1$  such that
\begin{gather}                                                                            \label{i31}
\Pr(T_n>n)\ge\sqrt{\frac{2}{\pi}}E_n\big(1-C_1(r_n+g_n^*)^{2/3}\big).
\end{gather}
On the other hand, there exists an absolute constant $C_2$ such that
\begin{gather}                                                                            \label{i32}
\Pr(T_n>n)\le\sqrt{\frac{2}{\pi}}E_n\big(1+C_2(r_n+g_n^*)^{2/3}\big),
\qquad\text{if}\quad
r_n+g_n^*\le1/24.
\end{gather}
In addition,  for $m\le n$,
\begin{gather}                                                                            \label{i33}
\Pr(T_n>m)\le\frac{4E_n}{B_m^{(n)}}
\end{gather}
provided that
$$
B_m^{(n)}:=\left(\sum_{k=1}^m\E X_{k,n}^2\right)^{1/2}\ge24(r_n+g_n^*).
$$
\end{theorem}
\begin{corollary}
\label{cor:asymp}
Under  conditions \eqref{i6}, \eqref{i7} and \eqref{i5} relation \eqref{i20} takes place.
\end{corollary}

Estimates \eqref{i31} and \eqref{i32} can be seen as an improved version of \eqref{i20}, 
with a rate of convergence. 
Moreover, the fact, that the dependence on $r_n$ and $g_n$ is expressed in a quite explicit way, 
is very important for our work  
\cite{DSW_prog} in progress, where we analyse unbounded random variables. 
In this paper we consider first-passage times of walks $S_n=X_1+X_2+\ldots+X_n$ for which the central limit theorem is valid but the Lindeberg condition may fail. We use Theorem~\ref{thm:main} to analyse the behaviour of triangular arrays obtained from $\{X_n\}$ by truncation.

\subsection{Triangular arrays of weighted random variables.}
Theorem~\ref{thm:main} and Corollary~\ref{cor:asymp} can be used in studying first-passage times of weighted sums of independent random variables. 

Suppose that we are given independent random variables 
$X_{1},X_2,\dots$ such that 
\begin{gather}                                                                           \label{ex1}
\E X_{i}=0
\quad\text{and}\quad
\Pr(|X_{i}|\le M_i)=1
\quad\text{for all }i\ge1,
\end{gather}
where $M_1,M_2,\dots$ are deterministic.
For each $n$ we consider a random walk
\begin{gather}                                                                            \label{ex2}
U_{k,n}:=u_{1,n}X_1+\dots+u_{k,n}X_k,\quad k=1,2,\dots,n,
\end{gather}
and let 
\begin{gather}                                                                               \label{ex3}
\tau_n:=\inf\{k\geq1:U_{k,n}\leq G_{k,n}\}
\end{gather}
be the first crossing over the moving boundary $\{G_{k,n}\}$ by the random walk $\{U_{k,n}\}$.
The main purpose of the present example is to study the asymptotic behaviour,
as $n\to\infty$, of the probability
\begin{gather}                                                                         \label{ex4}
\Pr(\tau_n>n)=\Pr\left(\min_{1\le k\le n}(U_{k,n}-G_{k,n})>0\right).
\end{gather}

We suppose that 
$\{u_{k,n}, G_{k,n}\}_{k=1}^n$ are deterministic real numbers such that
\begin{gather}                                                                            \label{ex6}
M:=\sup_{k,n\ge1}\left(|u_{k,n}|M_k+|G_{k,n}|\right)<\infty
\end{gather}
and
\begin{gather}                                                                            \label{ex7}
\sigma_n^2:=\sum_{k=1}^n u_{k,n}^2\E X_k^2\to\infty.
\end{gather}
Moreover, we assume that 
\begin{gather}                                                                            \label{ex5}
u_{k,n}\to u_k\quad\text{and}\quad G_{n,k}\to g_k
\quad\text{for every }k\ge1.
\end{gather}
\begin{corollary}                                                             	\label{cor:ex0}
Assume that the distribution functions of all $X_k$ are continuous. Then, under assumptions \eqref{ex1}, \eqref{ex6},
\eqref{ex7} and \eqref{ex5},
\begin{equation}                                                                        \label{ex8}
\sigma_n\Pr(\tau_n>n)\to\sqrt{\frac{2}{\pi}}
\E[-U_\tau]\in[0,\infty),
\end{equation}
where
\begin{gather}                                                                            \label{ex9}
U_{k}:=u_{1}X_1+\dots+u_{k}X_k
\quad\text{and}\quad
\tau:=\inf\{k\geq1:U_{k}\leq g_k\}.
\end{gather}
\end{corollary}
	
It follows from condition \eqref{ex5} that random walks $\{U_{k,n}\}$ introduced in 
\eqref{ex2} may be considered as perturbations of the walk $\{U_{k}\}$ defined in \eqref{ex9}. Thus, we see from  \eqref{ex8} that the influence of perturbations on the behavior of the probability $\Pr(\tau_n>n)$ is concentrated in the $\sigma_n$.

\begin{example}
\label{ex:Gaposhkin}
As an example we consider the following method of summation, which has been suggested by Gaposhkin~\cite{Gaposhkin}.
Let $f:[0,1]\mapsto\mathbb{R}^+$ be a non-degenerate continuous function. For random variables $\{X_k\}$ define
$$
U_k(n,f):=\sum_{j=1}^k f\left(\frac{j}{n}\right)X_j,
\quad j=1,2,\ldots,n.
$$
This sequence can be seen as a stochastic integral of $f$ with respect to the random walk $S_k=X_1+X_2+\ldots X_k$ normalized by $n$.

We assume that the random variables $\{X_k\}$ are independent and identically distributed. Furthermore, we assume that $X_1$ satisfies \eqref{ex1} and that its distribution function is continuous. In this case
$$
\sigma^2_n(f):=\frac1{n}\E X_1^2
\sum_{j=1}^n f^2\left(\frac{j}{n}\right)\to \sigma^2(f):=\E X_1^2\int_0^1f(t)dt>0.
$$
From Corollary~\ref{cor:ex0} with $u_{k,n}:=f\left(\frac{j}{n}\right)\to f(0)=:u_k$, $G_{k,n}\equiv0$ and $\sigma_n:=\sqrt{n}\sigma_n(f)$
we immediately obtain
\begin{equation}                                                                        \label{ex2.2}
\sqrt{n}\Pr\left(\min_{k\le n}U_k(n,f)>0\right)
\to \sqrt{\frac{2}{\pi}}\frac{f(0)}{\sigma(f)}
\E[-S_\tau]\in[0,\infty),
\end{equation}
where
\begin{gather}                                                                            \label{ex16}
S_{k}:=X_1+\dots+X_k
\quad\text{and}\quad
\tau:=\inf\{k\geq1:S_{k}\leq 0\}.
\end{gather}
\hfill$\diamond$
\end{example}

Clearly, \eqref{ex2.2} gives one exact asymptotics only when $f(0)>0$. The case $f(0)=0$ seems to be much more delicate. If $f(0)=0$ then one needs an information on the behaviour of $f$ near zero. If, for example, $f(t)=t^\alpha$ with some $\alpha>0$ then, according to Example 12 in \cite{DSW16}, 
$$
\Pr\left(\min_{k\le n}U_k(n,f)>0\right)
=\Pr\left(\min_{k\le n}\sum_{j=1}^k j^\alpha X_j>0\right)
\sim\frac{Const}{n^{\alpha+1/2}}.
$$

Now we give an example of application of our results to study of transition phenomena.

\begin{example}                                                           \label{ex:regression}
Consider an autoregressive sequence 
\begin{align}                                                                           \label{ex11}
&U_{n}(\gamma)=\gamma U_{n-1}(\gamma)+X_n,\ n\ge0,
\quad n=1,2,\dots,
\quad\text{where}\quad
U_0(\gamma)=0,
\end{align}
with  a non-random  $\gamma=\gamma_n\in(0,1)$ and with independent, identically distributed innovations
$X_1,X_2,\dots$. 
As in the previous example, we assume that $X_1$ satisfies \eqref{ex1} and that its distribution function is continuous.
Consider the exit time
$$
T(\gamma):=\inf\{n\ge1: U_n(\gamma)\le0\}.
$$
We want to understand the behavior of the probability
$\Pr(T(\gamma)>n)$ in the case when $\gamma=\gamma_n$ depends on $n$ and
\begin{gather}                                                                            \label{ex13}
\gamma_n\in(0,1)\quad\text{and}\quad
\sup_n n(1-\gamma_n)<\infty.
\end{gather}

We now show that the autoregressive sequence $U_n(\gamma)$ can be transformed to a random walk, which satisfies the  conditions of Corollary~\ref{cor:ex0}.
First, multiplying  \eqref{ex11} by $\gamma^{-n}$, we get 
$$
U_n(\gamma)\gamma^{-n}=U_n(\gamma)\gamma^{-(n-1)}+X_n\gamma^{-n}=
\sum_{k=1}^n\gamma^{-k}X_k,\quad n\ge1.
$$
Thus, for each $n\ge1$, 
\begin{align}                                                                                 \label{ex18}
\{T(\gamma_n)>n\}
=\left\{\sum_{j=1}^k \gamma_n^{-j}X_j>0\ \ \text{ for all }\ k\le n\right\}.
\end{align}
Comparing \eqref{ex18} with \eqref{ex2} and \eqref{ex4},
we see that we have a particular case of the model in Corollary~\ref{cor:ex0} with $u_{k,n}=\gamma_n^{-k}$ and $G_{k,n}=0$.  Clearly, $u_{k,n}\to1$ for every fixed $k$.
Furthermore, we infer from \eqref{ex13} that
$$                                                                                 
\gamma_n^{-n}=e^{-n\log \gamma_n}=e^{O(n|\gamma_n-1|)}=e^{O(1)}
$$
and
$$
\sigma_n^2(\gamma_n):=\frac{\gamma_n^{-2n}-1}{1-\gamma_n^2}
=\gamma_n^{-2}+\gamma_n^{-4}+\dots+\gamma_n^{-2n}
=ne^{O(1)}.
$$
These relations imply that \eqref{ex5} and \eqref{ex6} are fulfilled. Applying Corollary \ref{cor:ex0}, we arrive at 
\begin{equation}                                                                        \label{ex15}
\sigma_n(\gamma_n)\Pr(T(\gamma_n)>n)\to \sqrt{\frac{2}{\pi\E X_1^2}}
\E[-S_\tau]\in(0,\infty),
\end{equation}
where $\tau$ is defined in \eqref{ex16}.
\hfill$\diamond$
\end{example}

\subsection{Discussion of the assumption \eqref{i5}}
Based on the validity of CLT and considerations in~\cite{DSW16} one can expect that the Lindeberg condition will again be sufficient. 
However the following example shows that this  is not the case and the situation is more complicated.
\begin{example}                                                                \label{Lind}
Let $X_2,X_3,\ldots$ and $Y_2,Y_3,\ldots$ be mutually independent random variables such that
\begin{gather}                                                                            \label{i21-}
\E X_k=\E Y_k=0,\  \E X_k^2=\E Y_k^2=1\quad\text{and}\quad
\Pr(|X_k|\le M)=1\ \text{for all }k\ge2
\end{gather}
for some finite constant $M$. 
It is easy to see that the triangular array
\begin{gather}                                                                    \label{i24+}
X_{1,n}:=\frac{Y_n}{\sqrt{n}},\ 
X_{k,n}:=\frac{X_k}{\sqrt{n}},\ k=2,3,\ldots,n;\ n>1
\end{gather}
satisfies the Lindeberg condition. Indeed, $\sum_{i=1}^n\E X_{i,n}^2=1$ and 
for every $\ee>\frac M{\sqrt{n}}$ one has
\begin{gather}                                                                    \label{i26}
\sum_{i=1}^n\E [X_{i,n}^2;|X_{i,n}|>\ee]=\E [X_{1,n}^2;|X_{1,n}|>\ee]
\le \E X_{1,n}^2=\frac {\E Y_{n}^2}n=\frac1n\to0
\end{gather}
due to the fact that $|X_{k,n}|\le\frac M{\sqrt{n}}$ for all $k\ge2$. 

We shall also assume that $g_{k,n}\equiv0$. 
For each $n>1$ let random variables $Y_n$ be defined as follows 
	\begin{gather}                                                                          \label{i25}
Y_n:=
\begin{cases}
N_n,\! \!&\text{with probability}\  p_n:=\frac1{2N_n^2},\\
0,\!\! & \text{with probability}\   1-2p_n,\\
- N_n ,\!\! & \text{with probability}\  p_n,
\end{cases}
\end{gather}
where $N_n\ge1$. Note that $\E Y_n=0$ and $\E Y_n^2=1$.

For every $n>1$ we set
\begin{gather}                                                                            \label{i21}
U_{n}:=X_2+X_3+\ldots+X_{n} 
\quad\text{and}\quad
\uu U_{n}:=\min_{2\le i\le n}U_i.
\end{gather}
It is easy to see that
$$
\{T_n>n\}=\left\{Y_n=N_n\right\}\cap
\left\{\uu U_{n}>-N_n\right\}.
$$
Noting now that $\uu U_{n}\ge -(n-1)M$, we infer that 
\begin{gather}                                                                    \label{i27-}
\{T_n>n\}=\{Y_n=N_n\},\quad\text{for any }N_n>(n-1)M.
\end{gather}
In this case we have
\begin{align}                                                                                                                   \label{i28-}
\nonumber
E_n&=\E[S_{n,n};T_n>n]
=\E\left[\frac{Y_n+U_{n}}{\sqrt{n}};Y_n=N_n\right]\\
\nonumber
&=\Pr(Y_n=N_n)\E\left[\frac{N_n+U_{n}}{\sqrt{n}}\right]                                                          
=\Pr(Y_n=n)\frac{N_n+\E U_{n}}{\sqrt{n}}\\
&=\Pr(Y_n=n)\frac{N_n}{\sqrt{n}}.
\end{align}
In particular, from \eqref{i27-} and \eqref{i28-} we conclude that
\begin{align*}                                                                        
\Pr(T_n>n)=\Pr(Y_n=n)=\frac{E_n\sqrt{n}}{N_n}<\frac{E_n\sqrt{n}}{M(n-1)}=o(E_n)
\end{align*}
provided that $N_n>(n-1)M$. 

This example shows that \eqref{i20} can not hold for all triangular arrays satisfying the Lindeberg condition. 
\hfill$\diamond$
\end{example}
We now construct an array, for which the assumption \eqref{i5} becomes necessary
for the validity of \eqref{i20}.
\begin{example}                                                                \label{Lind2}
We consider again the model from the previous example and assume additionally
that the variables $X_2,X_3,\dots$ have the Rademacher distribution, that is,
$$
\Pr(X_k=\pm1)=\frac{1}{2}.
$$
Finally, in order to have random walks on lattices, we shall assume that $N_n$ is a natural number. 

It is then clear that $r_n:=\frac{N_n}{\sqrt{n}}$ is the minimal deterministic number such that
$$
\max_{k\le n}|X_{k,n}|\le r_n.
$$

As in Example~\ref{Lind}, we shall assume that $g_{k,n}\equiv0$.

In order to calculate $E_n$ we note that
\begin{align*}                                                                                                                                              
E_n=\E[S_{n,n};T_n>n]
&=\Pr\left(X_{1,n}=r_n\right)
\E\left[r_n+\frac{U_{n}}{\sqrt{n}};r_n+\frac{\uu U_{n}}{\sqrt{n}}>0\right]\\
&=\Pr\left(X_{1,n}=r_n\right)\frac{1}{\sqrt{n}}
\E\left[N_n+U_{n};N_n+\uu U_{n}>0\right].
\end{align*}
It is well known that for $m\ge1$ the sequence $(N+U_m){\rm 1}_{\{N+\underline{U}_m>0\}}$ is a martingale with $U_1=\underline{U}_1=0$. This implies that
$$
\E[N+U_m;N+\underline{U}_m>0]=N\quad\text{for all}\quad m,N\ge1.
$$
Consequently,
\begin{equation}
\label{i23}
E_n=p_n\frac{N_n}{\sqrt{n}}=p_n r_n.
\end{equation}

Furthermore,
\begin{align*}
\Pr(T_n>n)=\Pr\left(X_{1,n}=r_n\right)
\Pr\left(\frac{N_n}{\sqrt{n}}+\frac{\uu U_{n}}{\sqrt{n}}>0\right)
=p_n\Pr(N_n+\uu U_{n}>0).
\end{align*}
Using the reflection principle for the symmetric simple random walk, one can show that
\begin{equation}
\label{i24-}
\Pr\left(N+\uu U_m>0\right)=\Pr(-N<U_m\le N)
\quad\text{for all}\quad m,N\ge1.
\end{equation}
Consequently, $\Pr(T_n>n)=p_n\Pr(-N_n<U_{n}\le N_n)$. Combining this equality with \eqref{i23}, we obtain
\begin{equation}
\label{i27}
\frac{\Pr(T_n>n)}{E_n}=
\frac{1}{r_n}\Pr\left(-r_n<\frac{U_{n}}{\sqrt{n}}\le r_n\right).
\end{equation}

Using the central limit theorem, one obtains
\begin{align}                                                                            \label{i24}
\Pr\left(-r_n<\frac{U_{n}}{\sqrt{n}}\le r_n\right)
\sim\Psi\left(r_n\right),
\end{align}
where
\begin{gather}                                                                            \label{a3}
\varphi(u):=\frac{1}{\sqrt{2\pi}}e^{-u^2/2}
\qquad\text{and}\qquad 
\Psi(x):=2\int_0^{x^+}\varphi(u)du.
\end{gather}

We will postpone the proof of \eqref{i24-} and \eqref{i24} till the end of the paper.
Assuming that \eqref{i24-} and \eqref{i24} are true, as a result we have
$$
\frac{\Pr(T_n>n)}{E_n}\sim\frac{\Psi\left(r_n\right)}{r_n}.
$$
Noting now that $\frac{\Psi(a)}{a}<2\varphi(0)=\sqrt{\frac{2}{\pi}}$ for every $a>0$, we conclude that the assumption $r_n\to0$ is necessary and sufficient for the validity of \eqref{i20}.
More precisely,
\begin{itemize}
 \item $\Pr(T_n>n)\sim\sqrt{\frac{2}{\pi}}E_n$ iff $r_n\to0$;
 \item $\Pr(T_n>n)\sim\frac{\Psi(a)}{a}E_n$ iff $r_n\to a>0$;
 \item $\Pr(T_n>n)=o(E_n)$ iff $r_n\to\infty$.
\end{itemize}

\hfill$\diamond$
\end{example}

\section{Proofs.}
In this section we are going to obtain estimates, which are valid for each fixed $n$. 
For that reason we will sometimes omit the subscript $n$ and introduce the following simplified notation: 
\begin{gather}                                                                            \label{a0}
T:=T_n,\quad 
X_k:=X_{k,n},\quad S_k:=S_{k,n},\quad g_k:=g_{k,n},\ \ 1\le k< n
\end{gather}
and
\begin{gather}                                                                       \label{a1}
\rho:=r_n+g_n^*,\quad B_k^2:= \sum_{i=1}^k\E X_{i}^2,\quad 
B_{k,n}^2:=B_n^2-B_k^2=1-B_k^2,\ \ 1\le k< n.
\end{gather}
\subsection{Some estimates in the central limit theorem} 
For every integer~\mbox{$1\le k\le n$} and every real $y$ define
\begin{gather}                                                                                        \label{a2}
Z_k:=S_k-g_k,\ \widehat{Z}_k:=Z_k{\bf 1}\{T>k\}
\ \text{and}\ 
Q_{k,n}(y):=
\mathbf{P}\Big(y+\min_{k \le j\leq n}(Z_j-Z_k)>0\Big).
\end{gather}

\begin{lemma}                                                                                                          \label{Arak}
For all $y\in R$ and for all $0\leq k<n$ with $B_{k,n}>0$
\begin{equation}                                                                      \label{a4}
\left|Q_{k,n}(y)-\Psi\Big(\frac{y}{B_{k,n}}\Big)\right|
\le \frac{C_0\rho}{B_{k,n}}\ind{y>0}, 
\end{equation}
where $C_0$ is an absolute constant.
\end{lemma}
\begin{proof}
For non-random real $y$ define
\begin{gather}                                                                        \label{a5}
q_{k,n}(y):=
\mathbf{P}\Big(y+\min_{k \le j\leq n}(S_j-S_k)>0\Big),\quad n>k\ge1.
\end{gather}
It follows from Corollary 1 in Arak~\cite{Arak75}
that there exists an absolute constant $C_A$ such that
\begin{equation}                                                                    \label{a6}
\left|q_{k,n}(y)-\Psi\Big(\frac{y}{B_{k,n}}\Big)\right|
\le\frac{C_A}{B_{k,n}}\max_{k<j\le n}\frac{\mathbf{E}|X_j|^3}{\mathbf{E}X_j^2}\le \frac{C_A r_n}{B_{k,n}},
\end{equation}
where maximum is taken over all $j$ satisfying $\mathbf{E}X_j^2>0$. In the second step  we have used the inequality 
$\E|X_j|^3\le r_n\E X_j^2$ which follows from 
\eqref{i5}.

We have from (\ref{a2}) that $|Z_k- S_k|=|g_k|\le g_n^*$.
Hence, for $Q_{k,n}$ and $q_{k,n}$ defined in \eqref{a2} and \eqref{a5}, we have
\begin{equation}                                                                                                    \label{a7}
q_{k,n}(y_-)\le Q_{k,n}(y)\le q_{k,n}(y_+),\quad\text{where}\quad y_\pm:=y\pm2g_n^*.
\end{equation}
Then we obtain  from~\eqref{a6} that  
\begin{equation}                                                                                                          \label{a8}
\left|q_{k,n}(y_\pm)-\Psi\Big(\frac{y_\pm}{B_{k,n}}\Big)\right|
\le \frac{C_A r_n}{B_{k,n}}.
\end{equation}

On the other hand,
it is easy to see from \eqref{a3} that
$$
\Big| \Psi\Big(\frac{y_\pm}{B_{k,n}}\Big)-\Psi\Big(\frac{y}{B_{k,n}}\Big)\Big|
\le\frac{2\ff(0)|y_\pm-y|}{B_{k,n}}=\frac{4\ff(0)g_n^*}{B_{k,n}}.
$$
Applying this inequality together with~\eqref{a7} and~\eqref{a8} we immediately obtain~\eqref{a4} for $y>0$
$C_0:=C_A+4\ff(0)$. 
For $y\le 0$ inequality~\eqref{a4} immediately follows since 
$Q_{k,n}(y)=0=\Psi(y)$.
\end{proof}

\begin{lemma}                                                                                    \label{L4}
If $1\le m\le n$, then 
\begin{gather}                                                                                        \label{m11}
\E S_m^+\ge\frac{3}{8}B_m-r_n.
\end{gather}
Moreover, for all $m$ satisfying $B_m\ge24(r_n+g_n^*)$ we have 
\begin{gather}                                                                                        \label{m12}
\Pr(T>m)\le3\frac{\E \widehat{Z}_m}{B_m}.
\end{gather}
\end{lemma}
\begin{proof}
We will use the following extension of the Berry-Esseen inequality due to Tyurin~\cite{Tyurin2010}: 
\[
\sup_{x\in\mathbb{R}}|\Pr( S_m>x)-\Pr(B_m\eta>x)|
\le 0.5606\frac{\sum_{j=1}^m\E|X_j|^3}{B_m^3}
\le 0.5606\frac{r_n}{B_m},
\]
when $B_m>0$.
Here $\eta$ is a random variable that  follows the standard normal distribution. 
This inequality implies that, for every $C>0$,
\begin{align*}
\E S_m^+&=\int_0^\infty\Pr(S_m>x)dx\ge\int_0^{CB_m}\Pr(S_m>x)dx\\
&\ge\int_0^{CB_m}\left(\Pr(B_m\eta>x)-0.5606\frac{r_n}{B_m}\right)dx
=B_m\E(\eta^+\wedge C)-0.5606Cr_n.
\end{align*}
Further,
\begin{align*}
\E(\eta^+\wedge C) 
=\int_0^\infty(x\wedge C)\varphi(x)dx
&=\int_0^C x\frac{1}{2\pi}e^{-x^2/2}dx+C\int_C^\infty\ff(x)dx\\
&=\ff(0)-\ff(C)+C\int_C^\infty\ff(x)dx.
\end{align*}
Taking here $C=1/0.5606$ and using tables of the standard normal distribution we conclude that $\E(\eta^+\wedge C) >0.375>\frac{3}{8}$ and \eqref{m11} holds.

Next, according to Lemma 25 in \cite{DSW16},
\begin{gather}                                                                                        \label{m14}
\E Z^+_m\Pr(T>m)\le\E \widehat{Z}_m,\qquad 1\le m\le n.
\end{gather}
Therefore, it remains to derive a lower bound for $\E Z^+_m$.
We first note that
\begin{gather*}                                                                                        \label{m14+}
S_m=Z_m+g_m\le Z_m^++g_m^+\le Z_m^++g_n^*.
\end{gather*}
Hence, $S_m^+\le Z_m^++g_n^*$ and, taking into account \eqref{m11}, we get
\begin{gather}                                                                                         \label{m15}
\E Z_m^+\ge \E S_m^+-g_n^*\ge\frac{3}{8}B_m-(r_n+g_n^*) .
\end{gather}
If $m$ is such that $\frac{B_m}{24}\ge r_n+g_n^*$, then we infer from \eqref{m14} and \eqref{m15} that
\begin{gather*}                                                                                  \label{m16}
\E \widehat{Z}_m\ge\E Z^+_m\Pr(T>m)\ge
\left(\frac{3}{8}B_m-(r_n+g_n^*)\right)\Pr(T>n)
\\
\ge
\left(\frac{3}{8}-\frac{1}{24}\right) B_m\Pr(T>m)
=\frac{1}{3} B_m\Pr(T>m).
\end{gather*}
Thus,  \eqref{m12} is proven.
\end{proof}

\subsection{Estimates for expectations of $\widehat{Z}_k$.}
\begin{lemma}
\label{L3}
Let $\alpha$ be a stopping time such that $1\le\alpha\le l\le n$ with probability one. Then
\begin{equation}                                                                           \label{m2}
\mathbf{E}\widehat{Z}_\alpha- \mathbf{E}\widehat{Z}_l
\le2g_n^* p(\alpha,l)
\quad\text{with}\quad
p(\alpha,l):=\mathbf{P}(\alpha<T,\alpha<l).
\end{equation}
Moreover,
\begin{align}                                                                           \label{m3}
\mathbf{E}\widehat{Z}_\alpha- \mathbf{E}\widehat{Z}_l
&\ge\E[X_T;\alpha<T\le l]-2g_n^*p(\alpha,l)
\ge-(2g_n^*+r_n)p(\alpha,l).
\end{align}

In addition, the  equality in (\ref{i20+}) takes place.
\end{lemma}
\begin{proof}
Define events
$$
A_1:=\{\alpha<T\le l\}
\quad\text{and}\quad
A_2:=\{\alpha<l<T\}.
$$
Then, clearly,
$
\{\alpha<T,\alpha<l\}=A_1\cup A_2.
$
Using Lemma 20 from \cite{DSW16}, we obtain
\begin{align}                                            \nonumber
\mathbf{E}\widehat{Z}_\alpha+\mathbf{E}[ S_T;T\le\alpha]&=
-\mathbf{E}[ g_\alpha;\alpha<T]\\
\nonumber
&=-\mathbf{E}[ g_\alpha;A_2]-\mathbf{E}[ g_l;\alpha=l<T]-\mathbf{E}[ g_\alpha;A_1],
\\                                                                                        \label{m6blue}
\quad \mathbf{E}\widehat{Z}_l+\mathbf{E}[ S_T;T\le l]&=
-\mathbf{E}[ g_l;T>l]=-\mathbf{E}[ g_l;A_2]-\mathbf{E}[ g_l;\alpha=l<T].
\end{align}
Thus,
\begin{gather}                                                                                          \label{m6}
\mathbf{E}\widehat{Z}_\alpha- \mathbf{E}\widehat{Z}_l
=\mathbf{E}[ S_T- g_\alpha;A_1]
+\mathbf{E}[ g_l- g_\alpha;A_2].
\end{gather}
Next, by the definition of $T$,
$$
g_T\ge S_T=S_{T-1}+X_T> g_{T-1}+X_T.
$$
Hence,
\begin{align*}
\mathbf{E}[ S_T- g_\alpha;A_1]
\le\mathbf{E}[ g_T- g_\alpha;A_1]
\le2g_n^*\mathbf{P}(A_1)
\end{align*}
and
\begin{align*}
\mathbf{E}[ S_T- g_\alpha;A_1]
&\ge\mathbf{E}[ g_{T-1}- g_\alpha+X_T;A_1]\\
&\ge
\mathbf{E}[ X_T;A_1]-2g_n^*\mathbf{P}(A_1)
\ge-(2g_n^*+r_n)\mathbf{P}(A_1).
\end{align*}
Furthermore,
\begin{gather*}                                                                                          \label{m7}
|\mathbf{E}[ g_n- g_\alpha;A_2]|\le
2g_n^*\mathbf{P}(A_2).
\end{gather*}
Plugging these estimates into \eqref{m6}, we arrive at desired bounds.

The  equality in (\ref{i20+}) follows from  (\ref{m6blue}) with $l=n$.
\end{proof}

For every $h>0$ define
\begin{gather}                                                                              \label{d1}
\nu(h):=\inf\{k\geq1:S_k\geq g_k+h\}=\inf\{k\geq1:Z_k\geq h\}. 
\end{gather}

\begin{lemma}                                                                                       \label{L5}
Suppose that 
 $m\le n$ is such that
the inequality \eqref{m12} takes place,
\begin{gather}                                                                                        \label{d2}
B_m\ge24g_n^*
\quad\text{and}\quad
h\ge6g_n^*.
\end{gather}
Then
\begin{align}                                                                                 \label{d3}
2\mathbf{E}\widehat{Z}_{\nu(h)\wedge m}
\le3\mathbf{E}\widehat{Z}_m\le4\mathbf{E}\widehat{Z}_n=4E_n,
\quad
\mathbf{P}(\widehat{Z}_{\nu(h)\wedge m}>0)\le \kk E_n,
\\                                                                                            \label{d4}
2\kk g_n^*E_n\ge\mathbf{E}\widehat{Z}_{\nu(h)\wedge m}-E_n
\ge \delta(h)-2\kk g_n^*E_n,
\end{align}
where
\begin{align}                                                                                            \label{d5}
0\ge \delta(h):=\E[X_T;n\ge T>\nu(h)\wedge m]\ge- \kk r_nE_n
\quad\text{and}\quad \kk:=\frac2h+\frac4{B_m}.
\end{align}
In particular,  \eqref{i33} takes place.
\end{lemma}
\begin{proof}
First, we apply Lemma \ref{L3} with $l=m$ and
$\alpha=\nu(h)\wedge m$. For this choice of the stopping time one has
\begin{align*}
p(\nu(h)\wedge m,m)&=\Pr\left(\nu(h)\wedge m<T,\nu(h)\wedge m<m\right)\\
&\le\Pr(\widehat{Z}_{\nu(h)\wedge m}\ge h)
\le\frac{\E \widehat{Z}_{\nu(h)\wedge m}}{h}.
\end{align*}
Plugging this bound into \eqref{m2} and using the inequality $h\ge 6g_n^*$, we get 
\begin{align*}                                                                                 \mathbf{E}\widehat{Z}_{\nu(h)\wedge m}-\mathbf{E}\widehat{Z}_m
\le\frac{2g_n^*}{h}\mathbf{E}\widehat{Z}_{\nu(h)\wedge m}
\le\frac{\mathbf{E}\widehat{Z}_{\nu(h)\wedge m}}3
\end{align*}
and hence
\begin{align}
\label{d6}
\frac23\mathbf{E}\widehat{Z}_{\nu(h)\wedge m}\le\mathbf{E}\widehat{Z}_m.
\end{align}

Next, we apply Lemma \ref{L3} with $l=n$ and $\alpha=m$. In this case
\mbox{$p(m,n)=\Pr(T>m)$} and we may use \eqref{m12}.
Substituting these estimates into \eqref{m2} and using \eqref{d2}, we obtain
\begin{align*}                                                                                
\mathbf{E}\widehat{Z}_m-\mathbf{E}\widehat{Z}_n
\le 2g_n^* \Pr(T>m)
\le\frac{6g_n^*}{B_m}\mathbf{E}\widehat{Z}_m
\le\frac14\mathbf{E}\widehat{Z}_m.
\end{align*}
Therefore,
\begin{align}
\label{d7}
\frac34\mathbf{E}\widehat{Z}_m\le\mathbf{E}\widehat{Z}_n.
\end{align}
We conclude from  \eqref{d6} and \eqref{d7} that the first relation in \eqref{d3} takes place. In particular, from \eqref{m12} and \eqref{d7} we get that \eqref{i33} holds under assumptions of Lemma~\ref{L5}.

At last, we are going to apply Lemma \ref{L3} with $l=n>m$ and
$\alpha=\nu(h)\wedge m$. For this choice of the stopping time one has
\begin{align}                                                                         \label{d8}
\nonumber
p(\nu(h)\wedge m,n)&=\Pr\left(T>\nu(h)\wedge m\right)
=\Pr(\widehat{Z}_{\nu(h)\wedge m}>0)\\
\nonumber
&\le\Pr(\widehat{Z}_{\nu(h)\wedge m}\ge h)+\Pr\left(T> m\right)\\                                                            
&\le\frac{\E \widehat{Z}_{\nu(h)\wedge m}}{h}+\frac{3\E \widehat{Z}_m}{B_m}
\le\frac{2E_n}{h}+\frac{4E_n}{B_m}=\kk E_n.
\end{align}
Plugging this bound into \eqref{m2} and \eqref{m3}, we immediately obtain \eqref{d4}. 
The second inequality in \eqref{d3} also follows from \eqref{d8}; and using \eqref{i5} 
together with \eqref{d8} we find~\eqref{d5}.

Thus, all assertions of Lemma \ref{L5} are proved.
\end{proof}

\subsection{Proof of Theorem~\ref{thm:main}.}

According to the representation (36) in \cite{DSW16}, 
\begin{align}                                                                       \label{d11}
\nonumber
\mathbf{P}(T>n)
&=\mathbf{E}\left[Q_{\nu(h)\wedge m,n}(Z_{\nu(h)\wedge m});T>\nu(h)\wedge m\right]\\
&=\mathbf{E}Q_{\nu(h)\wedge m,n}(\widehat Z_{\nu(h)\wedge m}).
\end{align}

\begin{lemma}                                                                       \label{L2}
Suppose that all assumptions of Lemma \ref{L5} are fulfilled and  that \mbox{$B_{m,n}>0$.} Then one has
\begin{align}
 \label{d13}
 \nonumber
\left|{\mathbf{P}}(T>n)-\mathbf{E}\Psi\Big(
\frac{\widehat Z_{\nu(h)\wedge m}}
{B_{{\nu(h)\wedge m},n}}\Big)\right|
&\le \frac{C_0\rho}{B_{m,n}}\mathbf{P}(\widehat Z_{\nu(h)\wedge m}>0)\\
&\le2\ff(0)\frac{1.3C_0\kk\rho E_n}{B_{m,n}} .
\end{align}
In addition,
\begin{gather}                                         \label{d14}
\E\Psi\Big(\frac {\widehat Z_{\nu(h)\wedge m}}{B_{{\nu(h)\wedge m},n}}\Big)
\le \frac{2\ff(0) E_n(1+2\kk g_n^*)}{B_{m,n}},
\\                                                                                     \label{d15}
\E\Psi\Big(\frac {\widehat Z_{\nu(h)\wedge m}}{B_{{\nu(h)\wedge m},n}}\Big)
\ge 2\ff(0) E_n\Big(1-\frac{ (r_n+h)^2}{6}-2\kk g_n^*-\kk r_n\Big).
\end{gather}
\end{lemma}
\begin{proof}
Using (\ref{a4}) with $y=\widehat Z_{\nu(h)\wedge m}$, we obtain the first inequality in (\ref{d13}) as a consequence of (\ref{d11}).
The second  inequality in (\ref{d13}) follows from (\ref{d3}).

 Next, it has been shown in \cite[p. 3328]{DSW16} that
\begin{gather}                                                                        \label{d16}
2\ff(0)a\ge\Psi(a)\ge 2\ff(0)a(1-a^2/6)\quad
\text{for all }a\ge0.
\end{gather}
Recall that $0\le z:=\widehat Z_{\nu(h)\wedge m}\le r_n+h$ and $B_n=1$. Hence, by (\ref{d16}),
\begin{gather}                                                                        \label{d18}
\Psi\Big(\frac z{B_{{\nu(h)\wedge m},n}}\Big)\le\Psi\Big(\frac z{B_{m,n}}\Big)
\le\frac{2\ff(0) z}{B_{m,n}},
\\                                                                         \label{d19}
\Psi\Big(\frac z{B_{{\nu(h)\wedge m},n}}\Big)\ge\Psi\Big(\frac z{B_n}\Big)
\ge\frac{2\ff(0) z}{B_n}\Big(1-\frac{ z^2}{6B_n^2}\Big)
\ge{2\ff(0) z}\Big(1-\frac{ (r_n+h)^2}{6}\Big).
\end{gather}
Taking mathematical expectations in (\ref{d18}) and (\ref{d19})
with $z=\widehat Z_{\nu(h)\wedge m}$, we obtain:
\begin{gather}                                                                        \label{d20}
\frac{2\ff(0) \E\widehat Z_{\nu(h)\wedge m}}{B_{m,n}}\ge
\E\Psi\Big(\frac {\widehat Z_{\nu(h)\wedge m}}{B_{{\nu(h)\wedge m},n}}\Big)
\ge{2\ff(0) \E\widehat Z_{\nu(h)\wedge m}}\Big(1-\frac{ (r_n+h)^2}{6}\Big).
\end{gather}

 Now~(\ref{d14}) and~(\ref{d15}) follow from~(\ref{d20}) 
together with~(\ref{d3}) and~(\ref{d4}).
\end{proof}

\begin{lemma}                                                                                       \label{L7}
Assume that $\rho\le1/64$. Then inequalities \eqref{i31} and \eqref{i32}
take place with some absolute constants $C_1$ and $C_2$.
\end{lemma}
\begin{proof}
Set
 \begin{equation}                                                                           \label{d21}
m:=\min\{j\le n:B_j\ge\frac32\rho^{1/3}\}
\quad\text{and}\quad
h:=\rho^{1/3}.
\end{equation}
Noting that $r_n\le\rho\le\rho^{1/3}/4^2$ we obtain
\begin{gather}                                                                               \label{d22}
B_m^2=B_{m-1}^2+\E X_m^2<\left(\frac32\rho^{1/3}\right)^2+r_n^2
\le\frac94\rho^{2/3}+\frac1{4^6}
<\frac17.
\end{gather}
Consequently, $B_{m,n}^2=1-B_m^2$ and
we have from  \eqref{d21} that 
 \begin{equation}                                                                           \label{d23}
B_{m,n}^2>\frac67,\quad 24\rho\le\frac{24}{4^2}\rho^{1/3}=\frac32\rho^{1/3}\le B_m,
\quad 6g_n<\frac{6}{4^2}\rho^{1/3}<\rho^{1/3}=h.
\end{equation}

Thus,  all assumptions of Lemmas~\ref{L5} and~\ref{L2} are satisfied. Hence,  Lemma~\ref{L2} implies that
\begin{gather}                                                                               \label{d24}
2\ff(0) E_n(1-\rho_1-\rho_2 -2\kk \rho)\le\mathbf{P}(T>n),
\\                                                                               \label{d25}
\mathbf{P}(T>n)\le2\ff(0) E_n(1+\rho_1)(1+2\kk \rho)(1+\rho_3),
\end{gather}
where we used that $2g_n^*+r_n\le2\rho$ and
\begin{gather}                                                                               \label{d26}
\rho_1:=1.3C_0\kk\rho,\quad\rho_2:=\frac{ (r_n+h)^2}{6},
\quad  \rho_3:=\frac1{B_{m,n}}-1.
\end{gather}
Now from \eqref{d5} and \eqref{d21} with $\rho^{1/3}\le1/4$ we have 
\begin{gather*}                                                                               \label{d26+}
\rho\kk=\frac{2\rho}h+\frac{4\rho}{B_m}
\le2\rho^{2/3}+\frac{4\rho^{2/3}}{3/2}<4.7\rho^{2/3},
\quad r_n+h\le\frac1{4^2}\rho^{1/3}+\rho^{1/3}.
\end{gather*}
Then, by \eqref{d22}, 
\begin{align*}
\frac1{B_{m,n}}=\frac{B_{m,n}}{B_{m,n}^2}=\frac{\sqrt{1-B_m^2}}{1-B_m^2}
\le\frac{1-B_m^2/2}{1-B_m^2}=1+\frac{B_m^2}{2B_{m,n}^2}
<1+1.4\rho^{2/3}.
\end{align*}
So, these calculations and \eqref{d26} yield
\begin{gather}                                                                               \label{d28}
\rho_1<5C_0\rho^{2/3},\quad\rho_2<0.2\rho^{2/3},
\quad  \rho_3<1.4\rho^{2/3},
\quad2\kk\rho<9.4\rho^{2/3}.
\end{gather}

Substituting \eqref{d28} into \eqref{d24} we obtain \eqref{i31} with any 
$C_1\ge 5C_0+9.6$. On the other hand
from  \eqref{d28} and \eqref{d25} we may obtain \eqref{i32} with a constant $C_2$
which may be calculated in the following way:
\begin{gather*}                                                                               \label{d25+}
C_2=\sup_{\rho^{1/3}\le1/4}\left[5C_0(1+2\kk \rho)(1+\rho_3)+9.4(1+\rho_3)+1.4\right]<\infty.
\end{gather*}
\end{proof}

Thus, when $\rho\le1/4^3$, the both assertions of Theorem~\ref{thm:main}
immediately follow from  Lemma \ref{L7}. But if $\rho>1/4^3$ then
 \eqref{i32} is valid with any $C_1\ge4^2=16$ because in this case right-hand side in \eqref{i32} is negative.

Let us turn to the upper bound~\eqref{i32}.
If $\rho\le\frac{1}{24}$ but $\rho>\frac{1}{64}$ then \eqref{i33} holds for $m=n$;
and as a result we have from \eqref{i33} with any $C_2\ge32/\ff(0)$ that 
$$
\Pr(T_n>n)\le 4E_n\le4^3E_n\rho^{2/3}\le2\ff(0)E_n(1+C_2\rho^{2/3})
\quad\text{for}\quad
\rho^{1/3}>1/4.
$$

So, we have proved all assertions of Theorem~\ref{thm:main} in all cases.
\subsection{Proof of Corollary~\ref{cor:ex0}.}  
In order to apply Corollary~\ref{cor:asymp} we introduce the following triangular array:
\begin{gather}                                                                                  \label{ex21}
X_{j,n}:=\frac{u_{j,n}X_j}{\sigma_n},\quad g_{j,n}:=\frac{G_{j,n}}{\sigma_n},
\quad 1\le j\le n,\ n\ge1.
\end{gather}
The assumptions in \eqref{ex6} and \eqref{ex7} imply that the array introduced in \eqref{ex21} satisfies \eqref{i5} and \eqref{i6}. Thus,
\begin{align*}
\Pr\left(\tau_n>n\right)=\Pr(T_n>n)
&\sim\sqrt{\frac{2}{\pi}}\E[S_{n,n}-g_{n,n};T_n>n]\\
&=\sqrt{\frac{2}{\pi}}\Bigl(\E[S_{n,n};T_n>n]
-g_{n,n}\Pr(T_n>n)\Bigr).
\end{align*}
Here we also used \eqref{i20+}.
Since $g_{n,n}\to0$, we conclude that
$$
\Pr\left(\tau_n>n\right)
\sim \sqrt{\frac{2}{\pi}}\E[S_{n,n};T_n>n].
$$
Noting that $S_{n,n}=U_{n,n}/\sigma_n$, we get
\begin{equation}
\label{ex2.1}
\Pr\left(\tau_n>n\right)
\sim \sqrt{\frac{2}{\pi}}\frac{1}{\sigma_n}
\E[U_{n,n};\tau_n>n].
\end{equation}
By the optional stopping theorem, 
$$
\E[U_{n,n};\tau_n>n]=-\E[U_{\tau_n,n};\tau_n\le n].
$$
It follows from \eqref{ex5} that, for every fixed $k\ge1$,
\begin{equation}
\label{ex19}
U_{k,n}\to U_k\ \text{a.s.}
\end{equation}
and, taking into account the continuity of distribution functions,
\begin{align}
\label{ex20}
\nonumber
\Pr(\tau_n>k)
&=\Pr(U_{1,n}>G_{1,n},U_{2,n}>G_{2,n},\ldots,U_{k,n}>G_{k,n})\\
&\hspace{1cm}\to \Pr(U_{1}>g_1,U_{2}>g_2,\ldots,U_{k}>g_k)
=\Pr(\tau>k).
\end{align}
Obviously, \eqref{ex20} implies that
\begin{equation}
\label{ex29}
\Pr(\tau_n=k)\to \Pr(\tau=k)\quad
\text{for every }k\ge1.
\end{equation}

Furthermore, it follows from the assumptions \eqref{ex1} and \eqref{ex6} that 
\begin{equation}
\label{ex25}
|U_{\tau_n,n}|\le M\quad\text{on the event }\{\tau_n\le n\}.
\end{equation}

Then, combining \eqref{ex19}, \eqref{ex29} and \eqref{ex25}, we conclude that
\begin{equation}
\label{ex22}
\E[U_{\tau_n,n};\tau_n\le k]
=\sum_{j=1}^k\E[U_{j,n};\tau_n=j]
\to \sum_{j=1}^k\E[U_{j};\tau=j]
=\E[U_\tau;\tau\le k].
\end{equation}
Note also that, by \eqref{ex25} and \eqref{ex20},
\begin{equation*}
\limsup_{n\to\infty}|\E[U_{\tau_n,n};k<\tau_n\le n]|
\le M\limsup_{n\to\infty}\Pr(\tau_n>k).
\end{equation*}
Therefore,
\begin{align}
\label{ex23}
\nonumber
\limsup_{n\to\infty} \E[U_{\tau_n,n};\tau_n\le n]
&\le \limsup_{n\to\infty} \E[U_{\tau_n,n};\tau_n\le k]
+\limsup_{n\to\infty}|\E[U_{\tau_n,n};k<\tau_n\le n]|\\
&=\E[U_\tau;\tau\le k]+M\Pr(\tau>k)
\end{align}
and
\begin{align}
\label{ex24}
\nonumber
\liminf_{n\to\infty} \E[U_{\tau_n,n};\tau_n\le n]
&\ge \liminf_{n\to\infty} \E[U_{\tau_n,n};\tau_n\le k]
-\limsup_{n\to\infty}|\E[U_{\tau_n,n};k<\tau_n\le n]|\\
&=\E[U_\tau;\tau\le k]-M\Pr(\tau>k).
\end{align}
Letting $k\to\infty$ in \eqref{ex23} and \eqref{ex24}, and noting that $\tau$ is almost surely finite, we infer that
$$
\E[U_{\tau_n,n};\tau_n\le n]\to \E[U_\tau].
$$
Consequently, by the optional stopping theorem, 
$$
\E[U_{\tau_n,n};\tau_n>n]=-\E[U_{\tau_n,n};\tau_n\le n]\to \E[-U_\tau].
$$
Plugging this into \eqref{ex2.1}, we obtain
the desired result.

\subsection{Calculations related to Example~\ref{Lind2}}
\begin{lemma}                                                                                       \label{Ex2}
	For the simple symmetric random walk $\{U_m\}$ one has 
$$
\Pr\left(N+\uu U_m>0\right)=\Pr(-N<U_m\le N)
\quad\text{for all}\quad m,N\ge1
$$
and
$$
\sup_{N\ge1}\left|\frac{\Pr(-N<U_{n}\le N)}{\Psi(N/\sqrt{n})}-1\right|\to0.
$$
\end{lemma}
\begin{proof}
By the reflection principle for symmetric simple random walks, 
\begin{equation*}
\Pr\left(N+U_{m}=k, N+\uu U_{m}\le 0\right)=
\Pr(U_m=N+k)\quad\text{for every }k\ge1.
\end{equation*}
Thus, by the symmetry of the random walk $U_m$,
$$
\Pr\left(N+U_{m}>0, N+\uu U_{m}\le 0\right)
=\Pr(U_m<-N)=\Pr(U_m>N).
$$
Therefore,
\begin{align*}
\Pr\left(N+\uu U_{m}>0\right)
&=\Pr\left(N+U_{m}>0\right)-\Pr\left(N+U_{m}>0, N+\uu U_{m}\le 0\right)\\
&=\Pr(U_m>-N)-\Pr(U_m>N)
=\Pr(-N< U_{m}\le N).
\end{align*}

We now prove the second statement. Recall that $U_n$ is the sum of $n-1$ independent, Rademacher distributed random variables. By the central limit theorem, $U_{n}/\sqrt{n-1}$ converges to the standard normal distribution.
Therefore, $U_{n}/\sqrt{n}$ has the same limit. This means that
$$
\varepsilon_n^2:=\sup_{x>0}|\Pr(-x\sqrt{n}<U_n\le x\sqrt{n})-\Psi(x)|
\to0.
$$
Taking into account that $\Psi(x)$ increases, we conclude that, for every $\delta>0$,
$$
\sup_{x\ge\delta}\left|\frac{\Pr(-x\sqrt{n}<U_n\le x\sqrt{n})}{\Psi(x)}
-1\right|\le\frac{\varepsilon_n^2}{\Psi(\delta)}.
$$
Choose here $\delta=\varepsilon_n$. Noting that 
$\Psi(\varepsilon_n)\sim 2\varphi(0)\varepsilon_n$, we obtain
$$
\sup_{N\ge\varepsilon_n\sqrt{n}}\left|\frac{\Pr(-N<U_n\le N)}
{\Psi(N/\sqrt{n})}-1\right|
\le\frac{\varepsilon_n^2}{\Psi(\varepsilon_n)}
\sim\frac{\varepsilon_n}{2\varphi(0)}\to0.
$$

It remains to consider the case $N\le\varepsilon_n\sqrt{n}$. Here we shall use the local central limit theorem. Since $U_n$ is $2$-periodic,
$$
\sup_{k:\ k\equiv n-1({\rm mod}2)}
|\sqrt{n-1}\Pr(U_n=k)-2\varphi(k/\sqrt{n-1})|\to0.
$$
Noting that 
$$
\sup_{k\le \varepsilon_n\sqrt{n}}
|\varphi(k/\sqrt{n-1})-\varphi(0)|\to0,
$$
we obtain
$$
\sup_{N\le \varepsilon_n\sqrt{n}}
\left|\frac{\sqrt{n-1}\Pr(-N<U_n\le N)}{2\varphi(0)m(n,N)}-1\right|\to0,
$$
where
$$
m(n,N)=\#\{k\in(-N,N]:\ k\equiv n-1({\rm mod}2)\}.
$$
Since the interval $(-N,N]$ contains $N$ even and $N$ odd lattice points,
$m(n,N)=N$ for all $n$, $N\ge 1$. Consequently,
$$
\sup_{N\le \varepsilon_n\sqrt{n}}
\left|\frac{\sqrt{n-1}\Pr(-N<U_n\le N)}{2\varphi(0)N}-1\right|\to0,
$$
It remains now to notice that
$$
\Psi(N/\sqrt{n})\sim \frac{2\varphi(0)N}{\sqrt{n}}
$$
uniformly in $N\le \varepsilon_n\sqrt{n}$.
\end{proof}


\end{document}